\documentclass[a4paper,11pt,final]{amsart}

\usepackage{amsmath}
\usepackage{mathtools}
\usepackage[all]{xy}
\usepackage[latin1]{inputenc}
\usepackage{varioref}
\usepackage{amsfonts}
\usepackage{amssymb}
\usepackage{bbm}
\usepackage{graphicx}
\usepackage{mathrsfs}

\usepackage[hypertexnames=false,backref=page,pdftex,
 	pdfpagemode=UseNone,
 	breaklinks=true,
 	extension=pdf,
 	colorlinks=true,
 	linkcolor=blue,
 	citecolor=red,
 	urlcolor=blue,
 ]{hyperref}

\newcommand{\sI}{{\mathcal I}}

\newcommand{\sK}{{\mathcal K}}

\newcommand{\sO}{{\mathcal O}}

\newcommand{\sX}{{\mathcal X}}
\newcommand{\sY}{{\mathcal Y}}

\newcommand{\scrH}{{\mathscr H}}

\newcommand{\C}{{\mathbb C}}

\newcommand{\HH}{{\mathbb H}}

\newcommand{\N}{{\mathbb N}}
\renewcommand{\P}{{\mathbb P}}
\newcommand{\Q}{{\mathbb Q}}
\newcommand{\R}{{\mathbb R}}

\newcommand{\Z}{{\mathbb Z}}

\newcommand{\gothm}{{\mathfrak m}}

\newcommand{\gothX}{{\mathfrak X}}

\newcommand{\Art}{\operatorname{Art}}

\newcommand{\chern}{{\rm c}}
\newcommand{\codim}{\operatorname{codim}}

\newcommand{\coker}{\operatorname{coker}}

\newcommand{\Gr}{{\rm Gr}}

\newcommand{\Hom}{\operatorname{Hom}}

\newcommand{\img}{\operatorname{im}}
\newcommand{\into}{{\, \hookrightarrow\,}}
\newcommand{\isom}{\cong}

\newcommand{\lt}{{\rm{lt}}}

\newcommand{\onto}{{{\twoheadrightarrow}}}

\newcommand{\para}{{$\mathsection\,$}}

\newcommand{\red}{{\operatorname{red}}}
\newcommand{\reg}{{\operatorname{reg}}}

\newcommand{\rk}{{\rm rk}}
\newcommand{\sing}{{\operatorname{sing}}}

\newcommand{\Spec}{\operatorname{Spec}}

\renewcommand{\to}[1][]{\xrightarrow{\ #1\ }}

\newcommand{\tensor}{\otimes}
\newcommand{\tOm}{\widetilde{\Omega}}

\newcommand{\tY}{{\widetilde{Y}}}
\newcommand{\tsY}{{\widetilde{\sY}}}

\newcommand{\vphi}{\varphi}
\newcommand{\vrho}{{\varrho}}

\newcommand{\ul}[1]{{\underline{#1}}}

\newcommand{\wt}[1]{{\widetilde{#1}}}

\newtheoremstyle{citing}
  {}
  {}
  {\itshape}
  {}
  {\bfseries}
  {\textbf{.}}
  {.5em}
  {\thmnote{#3}}

\theoremstyle{plain}

\newtheorem{Thm}[subsection]{Theorem}
\newtheorem{theorem}[subsection]{Theorem}

\theoremstyle{definition}

\newtheorem{Cor}[subsection]{Corollary}
\newtheorem{definition}[subsection]{Definition}
\newtheorem{example}[subsection]{Example}
\newtheorem{lemma}[subsection]{Lemma}

\newtheorem{Prop}[subsection]{Proposition}
\newtheorem{Que}[subsection]{Question}
\numberwithin{equation}{section}

\theoremstyle{remark}

\newtheorem{remark}[subsection]{Remark}

{\theoremstyle{citing}
\newtheorem*{custom}{}}

\newcommand{\Def}{{\operatorname{Def}}}

\newcommand{\tomg}{\widetilde{\omega}}
\newcommand{\tom}{\tilde{\omega}}

\title[Deformations of Lagrangian subvarieties]{Deformations of Lagrangian subvarieties of holomorphic symplectic manifolds}

\author{Christian Lehn}
\address{Christian Lehn\\Institut f\"ur Algebraische Geometrie\\
Gottfried Wilhelm Leibniz Universit\"at Hannover\\Welfengarten 1\\30167 Hannover\\Germany}
\email{lehn@math.uni-hannover.de}

\begin{document}
\thispagestyle{empty}

\begin{abstract}
We generalize Voisin's theorem on deformations of pairs of a symplectic manifold and a Lagrangian submanifold to the case of Lagrangian normal crossing subvarieties. Partial results are obtained for arbitrary Lagrangian subvarieties. We apply our results to the study of singular fibers of Lagrangian fibrations.
\end{abstract}



\maketitle

\setlength{\parindent}{0em}
\setcounter{tocdepth}{1}

\tableofcontents

\section*{Introduction}

In \cite{Vo92} Voisin studied deformations of pairs $Y\subset X$ where $X$ is an irreducible symplectic manifold and $Y$ a complex Lagrangian submanifold. She showed that, roughly speaking, deformations of $X$ where $Y$ stays a complex submanifold are exactly those deformations, where $Y$ stays Lagrangian. 
We generalize Voisin's results to Lagrangian subvarieties with normal crossings. 

Let $\pi:\gothX \to M=\Def(X)$ be the universal deformation of $X$. By the Bogomolov-Tian-Todorov theorem, see \cite{Bo78,Ti,To89}, we know that $M$ is smooth. Let $\omega \in R^2\pi_*\C_\gothX\tensor \sO_M$ be a class restricting to a symplectic form on the fibers of $\pi$. For a subvariety $i: Y\into X$ denote by $\Def^\lt(i)$ the base of the universal locally trivial deformation of $i$ and by $p:\Def^\lt(i)\to M$ the forgetful map. Then we have

\begin{custom}[Theorem \ref{thm main}]
Let $i:Y \into X$ be a normal crossing Lagrangian subvariety in a compact irreducible symplectic manifold $X$, let $\nu:\tY\to Y$ be the normalization and denote $j=i\circ\nu$. Consider the subspaces
\[
\begin{xy}
\xymatrix@R=0em{
M_Y := \img(\Def^\lt(i) \to[p] M) \textrm{ and } M_Y' := \left\{t \in M : j^*\omega_t  = 0\right\}\\
}
\end{xy}
\]
of $M$. Then $M_Y'=M_Y$ and this space is smooth of codimension
\[
\codim_M M_Y = \codim_M M'_Y = \rk\left(H^2(X,\C) \to[j^*] H^2(\tY,\C)\right)
\]
in $M$.
\end{custom}

The space $M_Y'$ can be thought of as parametrizing those deformations for which $Y$ remains Lagrangian. We are especially interested in the space $M_Y$; it parametrizes deformations of $X$ such that $Y$ deforms along with it in a locally trivial manner, so in particular, keeping its singularities. We interpret this space as an invariant of the singularities of $Y$. Therefore, considering locally trivial deformations is not a restriction but has a geometric meaning. Note that if $Y$ is smooth, then every deformation is locally trivial. This is why the above theorem is a generalization of \cite[0.1 Th\'eor\`eme]{Vo92}.

Many of the intermediate steps in the proof of Theorem \ref{thm main} are essentially as in \cite{Vo92}, but for the smoothness of $M_Y$ we have to argue differently. For this, we develop ideas of Ran \cite{Ra92Lif}, \cite{Ra92Def} by exploiting the interplay between deformation theory and Hodge theory. 

This is also where there normal crossing hypothesis comes from. We show in Proposition \ref{prop omega is normal} that locally trivial deformations of the Lagrangian subvariety $Y$ inside $X$ are determined by the sheaf $\tOm_Y$. Its relation to Hodge theory is specific to the normal crossing case. Easy examples show that this is no longer the case for other types of singularities, see Example \ref{example k3}.
The necessary tools to apply Hodge theoretical arguments over an Artinian base were developed in \cite{CL12}. 

As in \cite{Vo92}, we deduce the following

\begin{custom}[Corollary \ref{cor main}]
Let $K:=\ker\left(H^2(X,\C) \to[j^*] H^2(\tY,\C)\right)$, let $q$ be the Beau\-ville-Bogomolov quadratic form and consider the period domain
\[
Q := \{\alpha \in \P(H^2(X,\C))\mid q(\alpha)=0, q(\alpha+\bar\alpha)>0\}
\]
of $X$. Then the period map $\wp:M\to Q$ identifies $M_Y$ with $\P(K)\cap Q$ locally at $[X]\in M$.
\end{custom}

A normal crossing Lagrangian subvariety in a symplectic manifold is quite special: it cannot have more than two local branches, see Lemma \ref{lemma lagrangian snc}. I am grateful to Claire Voisin for pointing this out. However, these are still the most important degenerations of Lagrangian submanifolds. For example, the majority of singular fibers of Lagrangian fibrations have normal crossings by the results of Hwang-Oguiso \cite{HO07}; so our results apply. A considerable part of Theorem \ref{thm main} holds true for arbitrary Lagrangian subvarieties. More precisely, we have
\[
(M_Y)_\red \subset M_Y' \textrm{ and } \codim_M M'_Y = \rk\left(H^2(X,\C) \to H^2(\tY,\C)\right),
\]
so that we can at least bound the codimension of $M_Y$, the space we are interested in, from below, see Theorem \ref{thm vormain}. This enables us to deduce results about Lagrangian fibrations.
\begin{custom}[Theorem \ref{theorem fibration}]
Let $X$ be an irreducible symplectic manifold and let $f:X\to B$ be a Lagrangian fibration. Then $X$ can be deformed, keeping the fibration, to an irreducible symplectic manifold $X'$ with a Lagrangian fibration $f':X'\to B'$ such that outside a codimension $2$ subset $Z\subset B'$, all singular fibers of $f'$ over the complement of $Z$ are of Kodaira type I, II, III or IV.
\end{custom}
This result is based on the Kodaira-type classification of singular fibers by Hwang-Oguiso, see section \ref{sec applications} for details. Similar results on Lagrangian fibrations were independently obtained by Justin Sawon \cite{Sa15} by completely different methods. 

Furthermore, the projectivity of arbitrary Lagrangian subvarieties of an irreducible symplectic manifold is shown. 
\begin{custom}[Theorem \ref{theorem line bundle}]
Let $i: Y \into X$ be a complex Lagrangian subvariety in an irreducible symplectic manifold. Then $Y$ is a projective algebraic variety.
\end{custom}

This is used to apply results from \cite{CL12}, but is also interesting in its own right. Again, the statement was known to Voisin in the smooth case.

Let us spend some words about the structure of this article. 
In section \ref{section lagrange} we show that a Lagrangian subvariety in an irreducible symplectic manifold is always projective. Section \ref{sec m} is basically an adaption of Voisin's results from \cite{Vo92} to our setting. The main new results of this article are contained in sections \ref{sec symplectic defo} and \ref{sec main}. In section \ref{sec symplectic defo} we prove smoothness of $\Def^\lt(i)$ in case $Y$ has normal crossings using the $T^1$-lifting principle. It also enables us to deduce that the canonical map $p:\Def^\lt(i) \to M$ has constant rank in a neighbourhood of the distinguished point, which implies the smoothness of the image $M_Y$.
Section \ref{sec main} finally puts together all previous theory to prove Theorem \ref{thm main} along the lines of Voisin's original argument with some additional input from Hodge theory and deformations of normal crossing varieties.
We give applications to Lagrangian fibrations in section \ref{sec applications}. First, we relate deformations of a singular fiber to deformations of the fibration and then we try to deform away from very singular fibers. Our results can be applied to most types of the general singular fibers of a Lagrangian fibration in the sense of Hwang-Oguiso \cite{HO07}. 

\section*{Notations and conventions}\label{subsec notation}

We work over the field $k=\C$ of complex numbers.
The term \emph{algebraic variety} will stand for a separated reduced $k$-scheme of finite type. In particular, a variety may have several irreducible components. Similarly, a \emph{complex variety} will be a separated reduced complex space. If there is no danger of confusion, we will drop the adjectives \emph{algebraic }respectively \emph{complex}. 
A variety $Y$ of equidimension $n$ is called a \emph{normal crossing variety} if for every closed point $y\in Y$ there is an $r\in\N_0$ such that $\widehat{\sO}_{Y,y} \isom k[[ y_1,\ldots,y_{n+1}]]/(y_1\cdot\ldots\cdot y_r)$. It is called a \emph{simple normal crossing variety} if in addition every irreducible component is nonsingular.

\subsection*{Acknowledgements} The large part of this work is the author's thesis. It is a pleasure to thank my advisor Manfred Lehn for his constant support, for suggesting many interesting problems and for thoroughly asking questions. Furthermore, I would like to thank Duco van Straten for some extremely helpful remarks, Yasunari Nagai and Keiji Oguiso for explaining Lagrangian fibrations, Klaus Hulek for helpful discussions on abelian fibrations, Jean-Pierre Demailly, Daniel Greb, Sam Grushevsky, Dmitry Kaledin, Thomas Peternell, S\"onke Rollenske, Justin Sawon, Gerard van der Geer, Claire Voisin and Kang Zuo for valuable discussions. I am also grateful to Justin Sawon for sending me his preprint \cite{Sa15}.

While working on this project, I benefited from the support of the DFG through the SFB/TR 45 ``Periods, moduli spaces and arithmetic of algebraic varieties'' and the DFG research grant Le 3093/1-1.

\section[Lagrangian subvarieties]{Projectivity of Lagrangian subvarieties}\label{section lagrange}
Let $X$ be an irreducible symplectic manifold, that is, a compact, simply connected K\"ahler manifold such that $H^0(X,\Omega_X^2) = \C \omega$ for a symplectic form $\omega$. A Lagrangian subvariety $i:Y \into X$ is a subvariety of dimension $\frac{\dim X}2$ such that $i^*\omega \in H^0(\Omega_Y^2)$ vanishes on $Y^\reg\subset Y$.

If $Y\subset X$ is a smooth Lagrangian subvariety, then by an argument of Voisin, $Y$ is projective even if $X$ is only K\"ahler, see \cite[Prop 2.1]{Ca06}. If $Y\subset X$ is a singular Lagrangian subvariety, it is natural to ask whether $Y$ is still projective. The following affirmative answer to this question  is used in the proof of our main theorem, but also interesting in its own right.
 The proof is a careful adaption of Voisin's argument to the singular setting. 
\begin{theorem}\label{theorem line bundle}
Let $i: Y \into X$ be a complex Lagrangian subvariety in an irreducible symplectic manifold. There is a line bundle $L$ on $Y$ such that $\chern_1\left(L\right) = i^* \lambda$ for some K\"ahler class $\lambda$ on $X$. In particular, $Y$ is a projective algebraic variety.
\end{theorem}
\begin{proof} Isomorphism classes of line bundles on $Y$ are classified by the group $H^1(Y,\sO^\times_Y)$, see \cite[Kap V, \para 3.2]{GR}. This cohomology group appears in the commutative diagram
\[
\begin{xy}
\xymatrix{
\ldots \ar[r] & H^1(Y,\sO_Y^\times) \ar[r] & H^2(Y, \Z) \ar[r] \ar[d] & H^2(Y,\sO_Y) \ar[r] & \ldots\\
& & H^2(Y,\C) \ar[r] & \HH^2(Y,\Omega_Y^\bullet) \ar[u] & \\
& & & \HH^2(Y,\Omega_Y^{\geq 1}) \ar[u] & \\
}
\end{xy}
\]
where the first line is the long exact sequence associated with the exponential sequence, see \cite[Kap V, \para 2.4]{GR}, and the right vertical column comes from the short exact sequence 
\[
0\to \Omega_Y^{\geq 1} \to \Omega_Y^\bullet \to \sO_Y \to 0.
\]
To obtain a holomorphic line bundle $L$ on $Y$ it is sufficient to find a class $\alpha \in H^2(Y,\Z)$, such that the image in $\HH^2(Y,\Omega_Y^\bullet)$ comes from $\HH^2(Y,\Omega_Y^{\geq 1})$. Such $L$ will have $\chern_1(L)=\alpha$.

As $X$ is K\"ahler, so is $Y$. Hence, the $H^k(Y,\Q)$ carry a mixed Hodge structure. Let us consider a resolution of singularities $\pi:\tY \to Y$. If $W_m \subset H^2(Y,\C)$ denotes the weight filtration, then $\pi^*$ factors as 
$$\pi^*:H^2(Y,\C) \onto H^2(Y,\C)/W_1 \into H^2(\tY,\C).$$ 
As $Y$ is Lagrangian, we have $i^*\omega = 0$ in $H^2(Y,\C)$ where $\omega \in H^0(X,\Omega_X^2)$ is the symplectic form on $X$. Indeed, it maps to $0$ in $H^2(\tY,\C)$ and as $X$ is smooth and morphisms of Hodge structures are strict, $i^*\omega$ is in $W_1$ if and only if it is zero. Consequently, also $H^{0,2}(X)$ maps to $0$ in $H^2(Y,\C)$.

Let us look at the composition $r:H^2(X,\C)\to H^2(Y,\C)\to \HH^2(Y,\Omega_Y^\bullet)$ and let $H \subset \HH^2(Y,\Omega_Y^\bullet)$ denote the image of $H^2(X,\R)$. 
By the Hodge-theoretic considerations above, the image of the K\"ahler cone $r(\sK_X)$ is open in $H$ and clearly, $r(H^2(X,\Q))$ is dense in $H$
so that there is in $0\neq \alpha'\in r(\sK_X) \cap r(H^2(X,\Q))$. Then a multiple $\alpha= m \cdot \alpha'$ is contained in $r(H^2(X,\Z) \cap i^*r(\sK_X)$ and we obtain a line bundle $L$ on $Y$ with the desired property by using the exponential sequence as explained above.

We conclude that $Y$ is projective by \cite[Chapter V, Corollary 4.5]{SCVVII}, see also \cite[3, Satz 1 and Satz~2]{Gra62}.
\end{proof} 

\section[Deformations]{Deformations of symplectic manifolds and Lagrangian subvarieties}\label{sec m}

As $H^0(X, T_{X}) = 0$ for an irreducible symplectic manifold $X$, the Kuranishi family $\pi:\gothX \to M=\Def(X)$ is universal at the point $0\in M$ corresponding to $X$. 
Close to $0\in M$ the fibers of $\pi$ are again irreducible symplectic manifolds, see \cite[\para 8]{Be83}. 
$M$ is known to be smooth by the Bogomolov-Tian-Todorov theorem \cite{Bo78,Ti,To89}, see also \cite[Thm 14.10]{GHJ}. Therefore, $\dim M = \dim T_{M,0}=h^1(T_X)=h^{1,1}(X)$.

\subsection{Deformations of closed immersions}\label{subsection defi}
We refer to \cite{Se} for an introduction to deformation theory. By $\Art_k$ we denote the category of local Artinian $k$-algebras with residue field $k$. 
Let $i:Y\into X$ be a closed immersion of algebraic $k$-schemes and suppose that $X$ is smooth and proper.
Let $R\in \Art_k$. A \emph{deformation of $i$} over $S=\Spec R$ is a diagram
\begin{equation}\label{eq defo morphism}
\xymatrix{
\sY \ar[rd] \ar@{^(->}[r] & \sX \ar[d]\\
& S, \\
}
\end{equation}
where $\sX\to S$ and $\sY\to S$ are flat and the fiber of \eqref{eq defo morphism} over $k=R/\gothm_R$ is isomorphic to $i:Y\into X$. 
Such a deformation is \emph{locally trivial} if for every $x\in X$, $y\in Y$ with $i(y)=x$ there is an open subset $U\subset X$  and such that $x \in U$ and the restriction $\sY\vert_U \into \sX\vert_U$ is a trivial deformation of $Y\cap U \into U$.

The tangent space to the deformation space of locally trivial deformations of a proper $k$-variety $Y$ is given by $H^1(T_Y)$, see \cite[Proposition 1.2.9]{Se}. Let $i:Y\into X$ be a closed immersion with smooth $X$ and let $T_{X}\langle Y\rangle$ be the kernel of the natural map $T_X \to i_*N_{Y/X}$, see \cite[3.4.4]{Se}. Then the tangent space to the deformation space for locally trivial deformations of $i$ is given by $H^1(T_{X}\langle Y\rangle)$, see \cite[Proposition 3.4.17]{Se}. 
More generally, given a diagram like \eqref{eq defo morphism}, we define $T_{\sX/S}\langle \sY\rangle$ by the exact sequence of sheaves on $\sX$
\begin{equation}\label{defo of morphism}
\begin{xy}
\xymatrix{
0\ar[r] & T_{\sX/S}\langle \sY\rangle \ar[r]&  T_{\sX/S} \ar[r]& N'_{\sY/\sX} \ar[r]& 0,\\
}
\end{xy}
\end{equation}
where
\begin{equation}\label{eq ns}
N'_{\sY/\sX}:=\ker(N_{\sY/\sX} \to T^1_{\sY/S})
\end{equation}
is the \emph{equisingular normal sheaf}.

Let $i:Y\into X$ be the inclusion of a closed subvariety in an irreducible symplectic manifold. Then, as a consequence of \cite{FK}, there is a universal locally trivial deformation of $i$ over a (germ of a) complex space $\Def^\lt(i)$.
 The inclusion $Y \into X$ gives a point $0\in \Def^\lt(i)$. By construction there is a forgetful morphism $p:\Def^\lt(i) \to M$ of complex spaces with $p(0)=0$. 
\begin{definition}\label{definition my}
We denote by $M_Y \subset M$ the image of $p$, that is, the smallest closed complex subspace such that $p$ factors through $M_Y \into M$.
\end{definition}
\subsection{The locus where a subvariety is Lagrangian}\label{subsec subspaces for lagrangian}
Let $i: Y \hookrightarrow X$ be the inclusion of a Lagrangian subvariety in an irreducible symplectic manifold $X$ of dimension $2n$.

We take a flat section $0\neq \omega \in R^0\pi_*\Omega^2_{\gothX/M}\into \scrH^2:=R^2\pi_*\ul\C_\gothX \tensor \sO_M$ 
and write $\omega_t:=\omega\vert_{X_t}$ for the symplectic form on the fiber $X_t=\pi^{-1}(t)$. We interpret $\omega_t$ as an element of $H^2(X,\C)$. Let $[Y]\in H^{2n}(X,\Z)$ denote the Poincar\'e dual of the fundamental cycle of $Y$. It has a unique flat lift to $\scrH^2$ and we denote by $[Y]_t$ the restriction of this lift to $\scrH^2_t = H^2(X_t,\C)$. Let us denote by $\nu :\tY \to Y$ a resolution of singularities and put $j=i\circ \nu$.	

\begin{definition}\label{definition subspaces}
With these notations following Voisin \cite{Vo92} we define 
\begin{equation}\label{m2}
M_Y' := \left\{t \in M \mid j^*\omega_t = 0 \textrm{ in } H^2(\tY,\C)\right\}.
\end{equation}
Clearly, this definition is independent of the resolution $\nu:\tY\to Y$.
\end{definition}
\begin{lemma} \label{lemma my' tangent}
The tangent space of $M_Y'$ at $0$ is given by
\begin{equation}\label{tmyprime}
T_{M'_Y,0} = \ker \left(H^1(T_X) \to[\omega'] H^1(\Omega_X) \to[j^* ] H^1(\Omega_\tY)\right)
\end{equation}
where $\omega'$ is the isomorphism induced by the symplectic form on $X$.
\end{lemma}
\begin{proof}
Locally at $0\in M$ the space $M'_{Y}$ is cut out by the equation $j_t^*\omega_t = 0$. Therefore, the tangent space at $0$ is given by the equation
\[
0=\left(\nabla j_t^*\omega_t\right)\vert_{t=0}=j^*\left(\nabla\omega_t\right)\vert_{t=0},
\]
where $\nabla$ is the Gau\ss-Manin connection. At $0$ it can be identified with the map
$H^0(\Omega_X^2) \to  \Hom(H^1(T_X),H^1(\Omega_X))$ 
given by cup product and contraction, which concludes the proof.
\end{proof}
\begin{Thm}\label{thm vormain}
Let $i:Y \into X$ be a Lagrangian subvariety in a compact irreducible symplectic manifold 
$X$, let $\nu:\tY\to Y$ be a resolution of singularities and denote $j=i\circ\nu$. Then $(M_Y)_\red \subset M_Y'$ and $M_Y'$ is smooth of codimension
\begin{equation}\label{codim prime}
\codim_M M'_Y = \rk\left(H^2(X,\C) \to[j^*] H^2(\tY,\C)\right)
\end{equation}
in $M$.
\end{Thm}
\begin{proof}
First assume that $Y$ is irreducible. 
Voisin shows in \cite[Propositions 1.2 and 1.7]{Vo92} that $M'_{Y}$ is a smooth submanifold of $M$ and that it coincides the Hodge locus $M_{[Y]}$ associated with the class $[Y]$ of $Y$ in $H^{2n}(X,\C)$, see \cite[Ch 5.3]{Vo2}. Smoothness of $Y$ is not needed for the first proposition, as its proof only uses the class of $Y$. The second one uses \cite[Lemme 1.5]{Vo92} which has to be replaced by Lemma \ref{lem rang} below.

Now let $Y=\cup_iY_i$ be a decomposition into irreducible components. Then set-theoretically
\begin{equation}\label{eq inclusions}
 M_Y \subset \bigcap_i M_{Y_i} \subset \bigcap_i M_{[Y_i]} = \bigcap_i M_{Y_i'} = M_Y',
\end{equation}
where the first inclusion is a consequence of \cite[Lemma 1.4]{CL12}, the inclusion $M_{Y_i}\subset M_{[Y_i]}$ is obvious and the equalities follow from 
the irreducible case and the definition of $M_Y'$. 
The statement about the codimension is deduced from the description \eqref{tmyprime} of the tangent space of $M_Y'$. 
\end{proof}
The following straight-forward generalization of \cite[Lemme 1.5]{Vo92} will complete the proof of Theorem \ref{thm vormain}. We include a full proof for convenience. Let $\mu:H^2(X,\C)\to H^{2+2n}(X,\C)$ be the map given by cup product with $[Y]$ and observe that it factors as
$H^2(X,\C)\to[j^*] H^2(\tY,\C) \to[j_*]  H^{2+2n}(X,\C)$.
\begin{lemma}\label{lem rang}
If $Y$ is irreducible, then
\[
\ker \left(H^2(X,\C) \to[\mu] H^{2n+2}(X,\C)\right) = \ker \left(H^2(X,\C)\to[j^*] H^2(\tY,\C)\right).
\]
\end{lemma}
\begin{proof}
We show equality of the respective kernels with real coefficients. From $\mu = j_* j^*$ we immediately have $\ker j^* \subset \ker \mu$. For the other inclusion we choose a K\"ahler class $\kappa \in H^2(X,\R)$. We have to show that $j_*$ is injective on $\img j^*$.

Assume $n=1$. As $\tY$ is connected, $H^2(\tY,\C)\isom \C$ and the map $j_*:H^2(\tY,\C)\to H^2(X,\C)$ is given by $1\mapsto [Y]$. As $X$ is K\"ahler, $[Y]\neq 0$. So $j_*$ is injective and the claim follows.

If $n\geq 2$, choose a K\"ahler class $\kappa \in H^2(X,\R)$. 
We may assume that $\tY \to Y$ is obtained by a sequence of blow-ups in smooth centers. Hence there is a K\"ahler class of the form $\wt\kappa = j^* \kappa - \sum_{i}\delta_i E_i \in H^2(\tY,\R)$ where the $E_i$ are exceptional divisors and $\delta_i \in \Q$ are positive.
We define a bilinear form 
\[
q(\alpha,\beta):=\int_\tY \wt\kappa^{n-2}.\alpha.\beta \qquad \alpha, \beta \in H^2(\tY,\C)
\]
on $H^2(\tY,\C)$. For $\alpha, \beta \in H^2(X,\R)$ this gives
\[
\begin{aligned}
q(j^*\alpha, j^*\beta) &= \int_\tY \wt\kappa^{n-2}.j^*(\alpha.\beta) = \int_X j_*\left(\wt\kappa^{n-2}.j^*(\alpha.\beta)\right)\\
& = \int_X\mu(\kappa^{n-2}).\alpha.\beta
 = \int_X\kappa^{n-2}.\mu(\alpha).\beta.
\end{aligned}
\]
So we see that if $\mu(\alpha)=0$, then $q(j^*\alpha, j^*\beta)=0$ for all $\beta \in H^2(X,\R)$. To conclude that $j^*\alpha = 0$ it would be sufficient to see that $q$ is non-degenerate on $\img j^* \subset H^2(\tY,\R)$. On the whole of $H^2(\tY,\R)$ the form $q$ is non-degenerate by the Hodge index theorem, see \cite[Thm 6.33]{Vo1}. Here we need that $\wt\kappa$ is a K\"ahler class. That $q$ remains non-degenerate on the subspace $\img j^*$ can also be deduced as follows. As $Y$ is Lagrangian $\img j^* \subset H^{1,1}(\tY,\R):=H^{1,1}(\tY) \cap H^2(\tY,\R)$ and on $H^{1,1}(\tY,\R)$ the form $q$ is non degenerate and has signature $(1,h^{1,1}-1)$. We know that $q(j^*\kappa,j^*\kappa) > 0$ and so $q$ is negative definite on $j^*\kappa^\perp$. Write $j^*\alpha = c \cdot j^*\kappa + \alpha'$ where $\alpha'\in j^*\kappa^\perp$. The decomposition shows that $\alpha' \in \img j^*$ as well. Then if $j^* \alpha \neq 0$ at least one of the numbers $q(j^*\alpha, j^*\kappa)$, $q(j^*\alpha, \alpha')$ is nonzero and so $\mu(\alpha)\neq 0$.
\end{proof}

\section{Normal crossing subvarieties}\label{sec symplectic defo}
Our next goal is to prove smoothness of the space $\Def^\lt(i)$ of locally trivial deformations of $i:Y\into X$, see Theorem \ref{thm mi smooth}, using a variant of the $T^1$-lifting principle. 
Smoothness plays an important role in the proof of our main result, Theorem \ref{thm main}. 
We start with some preliminary considerations on normal crossing varieties.

\begin{definition}\label{definition tom}
Let $f:\sY \to S$ be a proper morphism of schemes. We define $\tau^k_{\sY/S} \subset \Omega^k_{\sY/S}$ to be the subsheaf of sections whose support is contained in the singular locus of $f$. 
We put $\tOm^k_{\sY/S}:=\Omega^k_{\sY/S}/\tau^k_{\sY/S}$. Clearly, the exterior differential makes $\tOm_{\sY/S}^\bullet$ into a complex.
\end{definition}
If $Y$ is a normal crossing $\C$-variety, then the natural map $\C \to \tOm_Y^\bullet$ is a resolution. Moreover, if $Y$ is proper, this complex can be used to define the mixed Hodge structure on $H^k(Y,\C)$ as it has been done in \cite{Fr} if $Y$ has simple normal crossings. For a locally trivial deformation $f:\sY\to S$ of a simple normal crossing variety over an Artinian base scheme, it has been shown in \cite{CL12} that $\tOm_{\sY/S}^\bullet$ is a resolution of the constant sheaf $f^{-1}\sO_S$ and the Hodge theoretic analogues of Friedman's results have been established. It is possible to extend these results to the normal crossing case, i.e. components are allowed to have self-intersections. For this, let $Y$ be a normal crossing variety. We need a semi-simplicial resolution $\xymatrix@C=1em{
\ldots  \ar[r]\ar@<0.5ex>[r]\ar@<-0.5ex>[r] & Y^{[1]} \ar@<0.5ex>[r]\ar@<-0.5ex>[r] & Y^{[0]} \ar[r]& Y \\ }$ which replaces the canonical one in the simple normal crossing case, see \cite[p. 77]{Fr} and \cite[4.4]{CL12}. 
For Lagrangian subvarieties the situation is very simple. I am grateful to Claire Voisin for this observation. Its proof is straightforward, cf. \cite[Lemma 5.3]{GLR}.
\begin{lemma}\label{lemma lagrangian snc}
If $Y\subset X$ is a Lagrangian subvariety with normal crossings in a symplectic manifold $X$, then locally there cannot be more than two components. \qed
\end{lemma}
\begin{remark}\label{remark only nc}
 We thus obtain a semi-simplicial resolution where $\nu:Y^{[0]}\to Y$ is the normalization, $Y^{[1]} :=\nu^{-1}(Y^\sing)$ and the morphisms $\xymatrix@C=1em{Y^{[1]} \ar@<0.5ex>[r]\ar@<-0.5ex>[r] & Y^{[0]}}$ are the inclusion and its composition with the canonical involution $\tau :Y^{[1]}\to Y^{[1]}$ exchanging the two branches. Using this resolution, the Hodge theoretic results from \cite{CL12} carry over to the normal crossing situation.
\end{remark}

Let $\sX\to S=\Spec R$ for $R\in \Art_k$ be a deformation of an irreducible symplectic manifold $X$ and let $\omega\in H^0(\Omega^2_{\sX/S})$ be a relative symplectic form.
\begin{lemma}\label{lemma deformation is lagrange}
Let $i:Y \into X$ be a normal crossing Lagrangian subvariety. If $\sY\into \sX$ is a locally trivial deformation of $i$ over $S$, then $\sY$ is Lagrangian with respect to the symplectic form $\omega$ on $\sX$.
\end{lemma}
\begin{proof}
Let $\tsY\to S$ be the locally trivial deformation of the normalization of $Y$ obtained from \cite[Lemma 4.5]{CL12}. Note that $Y$ is projective by Theorem \ref{theorem line bundle}, so Lemma \cite[Lemma 4.5]{CL12} can be applied. 
As $Y$ has normal crossings, $\tsY\to S$ is smooth and $H^0(\Omega^2_{\sX/S})$ and $H^0(\Omega^2_{\tsY/S})$ are free $\sO_S$ modules by \cite[Th\'eor\`eme 5.5]{De68}.
Therefore, the pullback $H^0(\Omega^2_{\sX/S})\to H^0(\Omega^2_{\tsY/S})$ has constant rank by \cite[Theorem 4.17]{CL12}. As $\rk (j^*\tensor\C)=0$ on the central fiber, $j^*$ is identically zero and thus $\sY$ is Lagrangian.
\end{proof}
\begin{lemma}\label{lemma two sequences}
Let $i:Y\into X$ be a Lagrangian subvariety in an irreducible symplectic manifold $X$, let $S=\Spec R$ where $R\in \Art_\C$ and let $\sY\into \sX$ is a locally trivial deformation of $i$ over $S$. Then the symplectic form $\omega\in H^0(\Omega^2_{\sX/S})$ induces a morphism between the exact sequences 
\begin{equation}\label{omegat sequence}
\begin{xy}
\xymatrix{
& \sI/\sI^2 \ar[r] \ar@{-->}[d]^{\omega^{-1}}&  \Omega_{\sX/S} \otimes \sO_\sY \ar[r]\ar[d]^{\omega^{-1}}& \Omega_{\sY/S} \ar@{-->}[d]^{\omega'}\ar[r]& 0\\
0\ar[r] & T_{\sY/S} \ar[r]&  T_{\sX/S} \otimes \sO_\sY \ar[r]^\alpha & N_{\sY/\sX} \ar[r]& T^1_{\sY/S} \ar[r]& 0.\\
}
\end{xy}
\end{equation}
\end{lemma}
\begin{proof}
Since $\omega$ is non-degenerate, the map $\omega^{-1}:\Omega_{\sX/S} \to T_{\sX/S}$ is an isomorphism. The composition $\vphi:\sI/\sI^2 \to N_{\sY/\sX}=\Hom(\sI/\sI^2,\sO_\sY)$ is given by $f\mapsto \{f, \cdot\}$ where $\{\cdot,\cdot\}$ is the Poisson bracket associated with $\omega$. So $\vphi=0$ and the restriction of $\omega^{-1}$ to $\sI/\sI^2$ factors through $T_{\sY/S}=\ker \alpha$. Once we have this, we obtain a morphism $\omega' : \Omega_{\sY/S} \to N_{\sY/\sX}$, as the first line of \eqref{omegat sequence} is exact, by lifting sections to $\Omega_{\sX/S} \otimes \sO_\sY$.
\end{proof}
It is well-known that if in the situation of the preceding lemma the morphism $f:\sY\to S$ is smooth, then $\omega$ gives an isomorphism $\Omega_{\sY/S} \to N_{\sY/\sX}$.
The following Proposition \ref{prop omega is normal} explains what happens for singular Lagrangian subvarieties.

\begin{Prop}\label{prop omega is normal}
Let $i:Y\into X$ be a Lagrangian subvariety in an irreducible symplectic manifold $X$, let $S=\Spec R$ where $R\in \Art_\C$ and let $\sY \into \sX$ be a locally trivial deformation of $i$ over $S$. Let $\omega':\Omega_{\sY/S} \to N_{\sY/\sX}$ be as in \eqref{omegat sequence} and let $N'_{\sY/\sX}$ be the equisingular normal sheaf defined in \eqref{eq ns}. Then the diagram 
\begin{equation}\label{eq omega normal bundle}
\begin{xy}
\xymatrix{
\Omega_{\sY/S} \ar[r]^\omega\ar[d]& N_{\sY/\sX}\\
\tOm_{\sY/S} \ar@{-->}[r]_{\exists \, \tomg}  & N'_{\sY/\sX}\ar@{^(->}[u]\\
}
\end{xy}
\end{equation}
can be completed and $\tomg:\tOm_{\sY/S}\to N'_{\sY/\sX}$ is an isomorphism.
\end{Prop}
\begin{proof}
As $Y$ is Lagrangian, it is of pure dimension and thus $\sO_Y$ and also $\sO_\sY$ have no embedded primes. Locally, the sheaf $N_{\sY/\sX}$ can be embedded in a locally free sheaf and thus it does not have any embedded primes either.
Hence, $\tau^1_{\sY/S}$ maps to zero and $\tomg$ exists. But as $\omega$ is an isomorphism at smooth points of $f$, the support of $\ker \omega$ is contained in the singular locus of $f$, hence $\ker \omega \subset \tau^k_{\sY/S}$ and $\tomg$ is injective. Moreover, $\tOm_{\sY/S}$ maps onto $\ker(N_{\sY/\sX} \to T^1_{\sY/S})$ by \eqref{omegat sequence}, hence is identified with $N'_{\sY/\sX}$. 
\end{proof}

\subsection{The $T^1$-lifting Principle} \label{subsec t1-lifting}
To prove smoothness of $\Def^\lt(i)$ we will use Ran's $T^1$-lifting principle \cite{Ra92Def,Ka92, Ka97}, we refer to \cite{mydiss} for a gentle introduction. The basic idea is that in order to prove smoothness of a deformation functor it suffices to show that the corresponding $T^1$-modules are locally free for every infinitesimal deformation over a local Artinian scheme. This is achieved by means of Hodge theory.
\begin{Thm}\label{thm mi smooth}
Let $Y$ be a Lagrangian normal crossing subvariety of an irreducible symplectic manifold $X$. Then the complex space $\Def^\lt(i)$ is smooth at $0$.
Moreover, $M_Y$ is smooth and $\Def^\lt(i) \to M_Y$ is a submersion.
\end{Thm}
\begin{proof}

We have to show that the $H^1(T_{\sX/S}\langle \sY\rangle)$ are free. The sheaf $T_{\sX/S}\langle \sY\rangle$ was defined in \eqref{defo of morphism}.
Let $i:Y\into X$ be the inclusion and let $\sY\into \sX$ be a locally trivial deformation of $i$ over $S = \Spec R$ for $R\in \Art_\C$. Consider the long exact sequence
\begin{equation}\label{defo die lange original}
0 \to H^0(T_{\sX/S}\langle \sY\rangle) \to  H^0(T_{\sX/S}) \to H^0(N'_{\sY/\sX}) \to H^1(T_{\sX/S}\langle \sY\rangle) \to \ldots\\
\end{equation}
obtained from the sequence \eqref{defo of morphism}. We transform this sequence using the isomorphism $T_{\sX/S} \isom \Omega_{\sX/S}$, Lemma \ref{lemma deformation is lagrange}, Proposition \ref{prop omega is normal} and $H^0(\Omega_{\sX/S})=0$ to obtain an exact sequence
\begin{equation}\label{defo die lange}
0\to H^0(\tOm_{\sY/S}) \to H^1(T_{\sX/S}\langle \sY\rangle) \to  H^1(\Omega_{\sX/S}) \to H^1(\tOm_{\sY/S}) \to\ldots\\
\end{equation}
Recall from Definition \ref{definition my} that we have a factorization $p:\Def^\lt(i) \to M_Y \into M$. However, in general it is not clear whether $\Def^\lt(i) \to M_Y$ is surjective, let alone submersive.
By Theorem \cite[Th\'eor\`eme 5.5]{De68} we know that $H^k(\Omega_{\sX/S})$ is free. By Theorem \ref{theorem line bundle} we know that $Y$ is a projective variety, so Theorem \cite[Theorem 4.13]{CL12} applies and $H^k(\tOm_{\sY/S})$ is free. Note that the results of \cite{CL12} carry over literally to the normal crossing case as was explained in Remark \ref{remark only nc}. Then by Theorem \cite[Theorem 4.22]{CL12} also the cokernel (and hence the kernel) of $H^k(\Omega_{\sX/S}) \to H^k(\tOm_{\sY/S})$ is free. From sequence \eqref{defo die lange} we deduce that all $H^k(T_{\sX/S}\langle \sY\rangle)$ are free and that all morphisms in \eqref{defo die lange} have constant rank. In particular, all morphisms in \eqref{defo die lange original} have constant rank. The $T^1$-lifting principle implies that $\Def^\lt(i)$ is smooth. 

So the canonical morphism $p:(\Def^\lt(i),0) \to (M,0)$ is just a holomorphic map between (germs of) complex manifolds. 
To prove the theorem it suffices to show that ts differential $Dp$ has constant rank in a neighbourhood of $0$.
This holds if the stalk of $\coker(p_*: T_{\Def^\lt(i)}\to p^* T_M)$ at $0$ is free. Freeness may be tested after completion, and then by the local criterion for flatness \cite[Thm A.5]{Se} we may test it for the truncations modulo powers of the maximal ideal. In other words we have to verify, given as above a locally trivial deformation $\sY \into \sX$ of $i$ over $S=\Spec R$ with $R\in \Art_\C$, that the map $H^1(T_{\sX/S}\langle \sY\rangle) \to H^1(T_{\sX/S})$ has constant rank. This was already noted in the first part of the proof.
\end{proof} 

\section{Codimension formula}\label{sec main}
Let $i: Y \hookrightarrow X$ be the inclusion of a Lagrangian subvariety in an irreducible symplectic manifold. In this section we show that if $Y$ has normal crossings, then the inclusion $M_Y\subset M_Y'$ from Theorem \ref{thm vormain} is an equality. In this way, we obtain a formula for the codimension of $M_Y$.

\begin{lemma}\label{lemma weight argument}
Suppose $Y$ has normal crossings. Then $$\ker\left(H^1(\Omega_X) \to[j^*] H^1(\Omega_\tY)\right) = \ker\left(H^1(\Omega_X) \to[i^*] H^1(\tOm_Y)\right),$$
where $\nu:\tY\to Y$ is the normalization.
\end{lemma}
\begin{proof}
As $j^* = \nu^* \circ \, i^*$ the inclusion $\supset$ is obvious. For the other direction it suffices to show that $\nu^*$ is injective on $\img i^*$. By Theorem \ref{theorem line bundle} the subvariety $Y$ is projective, hence by \cite{De71,De74} there is a functorial mixed Hodge structure on $H^k(Y,\C)$ for every $k$. We denote by $F^\bullet$ the Hodge filtration on $H^2(Y)$ and by $W_\bullet$ the weight filtration. As a special case of \cite[Cor 4.16]{CL12}, we deduce that 
\[
H^1(\tOm_Y)= \Gr_F^1 H^2(Y) = F^1H^2(Y) / F^2H^2(Y).
\]
Let $\begin{xy}
\xymatrix@C=1em{
\ldots  \ar[r]\ar@<0.5ex>[r]\ar@<-0.5ex>[r] & Y^{1} \ar@<0.5ex>[r]\ar@<-0.5ex>[r] & Y^{0} \ar[r]& Y \\
}
\end{xy}$ be the canonical semi-simplicial resolution in the simple normal crossing case, see e.g. \cite[4.8]{CL12}, or the one from Remark \ref{remark only nc} in the normal crossing case. Note that $\tY=Y^{0}$. Consider the weight spectral sequence associated with the first graded objects of the Hodge filtration given by
\begin{equation}\label{gr spec seq}
E_1^{r,s}= H^s(Y^{r}, \Omega_{Y^{r}}^1) \Rightarrow H^{r+s}(Y,\tOm^1_Y)
\end{equation}
By \cite[Thm 3.12 (3)]{PS} it degenerates on the same level as the weight spectral sequence, which is known to degenerate at $E_2$. The differential $d_1:E_1^{0,1} \to E_1^{0,1}$ is given by $\delta:H^1(\Omega_{Y^{0}})\to H^1(\Omega_{Y^{1}})$ and degeneration at $E_2$ tells us that
\begin{align*}
\Gr_2^W\Gr_F^1 H^2(Y) &= F^1 H^2(Y)/(W_1F^1H^2(Y) + F^2 H^2(Y)) = E_\infty^{0,1} = E_2^{0,1}\\&= \ker\left(H^1(\Omega_{Y^{0}})\to H^1(\Omega_{Y^{1}})\right).
\end{align*}
In other words, as $W_2 \Gr_F^1 H^2(Y) = \Gr_F^1 H^2(Y) = H^1(\tOm_Y)$ there is an exact sequence
\[
0\to W_1 \Gr_F^1 H^2(Y) \to H^1(\tOm_Y) \to[\nu^*] H^1(\Omega_{Y^{0}})\to H^1(\Omega_{Y^{1}}),
\]
so that $\ker \nu^* =  W_1 \Gr_F^1 H^2(Y)$. But the Hodge structure on $H^2(X,\C)$ has pure weight two because $X$ is smooth. In particular, $W_1\Gr_F^1 H^2(X) = 0$. Morphisms of mixed Hodge structures are strict with respect to both filtrations, so we have
\[
0 = i^*(W_1\Gr_F^1H^2(X)) = \img i^* \cap W_1\Gr_F^1 H^2(Y) = \img i^*\cap \ker \nu^*
\]
hence $\nu^*$ is injective on $\img i^*$ and we deduce $\ker i^* = \ker j^*$ completing the proof.
\end{proof}
The following lemma generalizes \cite[Lem 2.3]{Vo92} to the normal crossing case.
\begin{lemma}\label{tisequal}
Suppose $Y$ has  normal crossings. Then we have $T_{M'_Y,0} = T_{M_Y,0}$ for the Zariski tangent spaces at $0 \in M_Y \cap M'_{Y}$.
\end{lemma}
\begin{proof}
By Lemma \ref{lemma my' tangent} the tangent space of $M_Y'$ at $0$ is
\[
T_{M'_Y,0}= \ker \left(j^* \circ \omega':H^1(X,T_X)\to  H^1(\tOm_Y)\right).
\]
By Lemma \ref{lemma weight argument} this equals $\ker \left(i^* \circ \omega':H^1(X,T_X)\to  H^1(\Omega_\tY)\right)$,
where $\tY\to Y$ is the normalization. On the other hand, $M_Y$ is the smooth image of $p:\Def^\lt(i)\to M$ so that
\begin{align*}
T_{M_Y,0} &=\img \left(p_*: T_{\Def^\lt(i),0}\to T_{M,0}\right)\\
 &= \img \left(H^1(X,T_{\sX/S}\langle \sY\rangle) \to H^1(X,T_X)\right) \\
 &= \ker\left(H^1(X,T_X)\to[\alpha] H^1(Y,N'_{Y/X})\right)
\end{align*}
where the third equality holds because the sequence \eqref{defo die lange original} is exact. 

By \eqref{omegat sequence} and Proposition \ref{prop omega is normal} we have a commutative diagram
\begin{equation*}
\begin{xy}
\xymatrix{
H^1(X,\Omega_{X})\ar[r]^{j^*}& H^1(Y,\tOm_Y) \ar[d]^\tom\\
H^1(X,T_X)\ar[u]_{\omega'}\ar[r]^\alpha& H^1(Y,N'_{Y/X})\\
}
\end{xy}
\end{equation*}
where the vertical maps are isomorphisms and this completes the proof.
\end{proof}
\begin{Thm}\label{thm main}
Let $i:Y \into X$ be a normal crossing Lagrangian subvariety in a compact irreducible symplectic manifold 
$X$, let $\nu:\tY\to Y$ be the normalization and denote $j=i\circ\nu$. Then $M_Y$ is smooth at $0$ of codimension
\begin{equation}\label{codim}
\codim_M M_Y = \rk\left(H^2(X,\C) \to[j^*] H^2(\tY,\C)\right)
\end{equation}
in $M$.
\end{Thm}
\begin{proof}
By Theorems \ref{thm vormain} and \ref{thm mi smooth} we have $M_Y \subset M_Y'$ and it suffices to show equality. This is deduced from $\dim M_Y \leq \dim M'_{Y} \leq \dim T_{M'_Y,0} = \dim T_{M_Y,0}$,
where the last equality comes from Lemma \ref{tisequal}, again by invoking smoothness of $M_Y$. 
\end{proof}
\begin{Cor}\label{cor main}
Let $K:=\ker\left(H^2(X,\C) \to[j^*] H^2(\tY,\C)\right)$, let $q$ be the Beau\-ville-Bogomolov quadratic form and consider the period domain
\[
Q := \{\alpha \in \P(H^2(X,\C))\mid q(\alpha)=0, q(\alpha+\bar\alpha)>0\}
\]
of $X$. Then the period map $\wp:M\to Q$ identifies $M_Y$ with $\P(K)\cap Q$ locally at $[X]\in M$.
\end{Cor}
\begin{proof}
As the period map identifies $M$ with $Q$ it suffices to show that $\wp(M_Y)=\P(K)\cap Q$. By  \cite[1.14]{Hu99}, $\P(K)\cap Q$ is the locus where $K^\perp \subset H^2(X,\C)$ remains of type $(1,1)$ and its codimension is $\dim K^\perp$. Note that $K^\perp \subset H^{1,1}(X)$ is defined over $\Z$ and therefore is spanned by the Chern classes of a collection of line bundles on $X$. 
By Lemma \ref{lemma deformation is lagrange} the subspace $K^\perp$ remains of type $(1,1)$ over $M_Y$. Hence, $\wp(M_Y) \subset \P(K)\cap Q$. Moreover, we have
\[
\begin{aligned} 
\codim_Q \wp(M_Y) &= \codim_M M_Y   = \rk\left(j^*:H^2(X,\C) \to H^2(\tY,\C)\right) \\
&= b_2(X) - \dim K = \dim K^\perp\\
& = \codim_Q \P(K)\cap Q.
\end{aligned}
\]
So both sets are equal.
\end{proof}

\section{Applications to Lagrangian fibrations}\label{sec applications}
In this section we give some applications of Theorems \ref{thm vormain} and \ref{thm main} to Lagrangian fibrations. Our main goal is to determine $\codim_M M_Y$. Let $X$ be an irreducible symplectic manifold. Recall that a \emph{Lagrangian fibration} is a morphism $f:X\to B$ with connected fibers to a normal projective variety $B$ such that the general fiber of $f$ is a Lagrangian subvariety.

Lagrangian fibrations are an important tool to study irreducible symplectic manifolds. It is conjectured that an arbitrary irreducible symplectic manifold can always be deformed to one that admits a Lagrangian fibration. Moreover, Matsushita has shown in a series of papers \cite{Matsus99a,Matsus00,Matsus01a,Matsus03} that every fibration of an irreducible symplectic manifold is a Lagrangian fibration and that the base $B$ resembles the projective space $\P^n$. The holomorphic Liouville-Arnol'd theorem shows that every smooth fiber is a complex torus, hence singular fibers enter the focus.

Hwang-Oguiso \cite{HO07} classified generic singular fibers of a Lagrangian fibration. For such a fiber they defined the \emph{characteristic cycle}, a (maybe infinite) cycle $\Theta$ of curves on the fiber, and they have shown that it is either a Kodaira singular fiber of an elliptic surface or an infinite chain of smooth rational curves intersecting transversally (so-called $I_\infty$-type). Locally, the fiber is isomorphic to $\Theta \times \C^{n-1}$ and the intersection graph of the fiber is a quotient of the graph of the characteristic cycle. The datum of the intersection graph of the fiber together with its local singularities is what we call \emph{fiber type}. 

In view of these classification results, Theorem \ref{thm main} applies to the majority of (reductions of) generic singular fibers of a Lagrangian fibration. Only fibers with characteristic cycle of Kodaira types II, III and IV are not normal crossing varieties; for those we have Theorem \ref{thm vormain}. Here it is important that we consider locally trivial deformations. It entails that the fiber type in the Hwang-Oguiso sense does not change so that we obtain an invariant of this fiber type, see Theorem \ref{theorem fibration}. 
Note that this is not in general the case for the characteristic cycle, see \cite[Proposition 5.3]{HO10} for an example.

\enlargethispage{\baselineskip}

\subsection{Deforming fibrations}\label{subsec fibration deforms}
We show first that if we deform a fiber of a fibration then also the fibration deforms, see Lemma \ref{lemma my in mf}. 
Let $f:X\to B$ be a Lagrangian fibration and assume that $B$ is  projective. 
Matsushita showed in \cite[Corollary 1.2]{Matsus09} that there is a smooth hypersurface $\Def(X,f) \subset M$ with a relative Lagrangian fibration extending $f$
\[
\begin{xy}\xymatrix{
\gothX \ar[dr]^\pi\ar[rr]^F && P \ar[dl]\\
&\Def(X,f) &
}\end{xy}
\]
where $\pi:\gothX \to \Def(X,f) $ is the restriction of the universal family to $\Def(X,f) $  and $P\to \Def(X,f) $  is a projective morphism. In particular, $F_t:\gothX_t \to P_t$ is a Lagrangian fibration and $F_0=f$. 
Let $T$ be a smooth fiber of $f$ and let $M_T\subset M$ be as in Theorem \ref{thm mi smooth}. Then $M_T=\Def(X,f)$ by \cite[Proposition 2.1(3)]{Matsus09}. 
The following lemma tells us that if the reduced fiber is preserved as a subvariety, then also the fibration is preserved.
\begin{lemma}\label{lemma my in mf}
Let $f:X\to B$ be a Lagrangian fibration, let $t\in B$ and let $Y=\left(X_t\right)_\red$ be the reduction of a fiber. Then we have $M_Y \subset \Def(X,f) $. Moreover, locally trivial deformations of $Y$ remain fibers.
\end{lemma}
\begin{proof}
By \ref{subsec fibration deforms} it is sufficient to show $M_Y \subset M_T$. Let $Y=\cup_{i\in I} Y_i$ be a decomposition into irreducible components.
As in \eqref{eq inclusions} we have $M_Y \subset \cap_i M_{[\sum_i n_i Y_i]}$  and for a smooth fiber $T$ of $f$ we have $\sum_i n_i[Y_i] = [T]$ and so that $M_{[\sum_i n_i Y_i]} = M_{[T]} = M_T$,
where the last equality is Voisin's theorem. Put together this gives $M_Y\subset M_T = \Def(X,f) $. As Lagrangian fibrations are equidimensional, the last claim follows from the Rigidity Lemma \cite[Lem 1.6]{KM}. 
\end{proof}

\subsection{Codimension estimates}\label{subsec codimension}
Let $X$ be an irreducible symplectic manifold and let $f:X\to B$ be a Lagrangian fibration. 
We put $Y=\left(X_t\right)_\red$ for $t \in D := \left\{t\in B: X_t \textrm{ is singular}\right\}$. The analytic subset $D$ is called the \emph{discriminant locus} of $f$. 
We know by \cite[Prop 4.1]{Hw} and \cite[Prop 3.1]{HO07} that $D$ is nonempty and of pure codimension one. 

Let $D_0\ni t$ be an irreducible component of $D$ and let $X_0:=X\times_{B}D_0 = f^{-1}(D_0)$. Let $Y=\cup_{i\in I} Y_i$ and $X_0=\cup_{j\in J} X_j$ be decompositions into irreducible components and consider the surjective map $j:I\to J$ mapping $i\in I$ to the unique $j=j(i)\in J$ with $Y_i \subset X_j$.

I am very grateful to Keiji Oguiso for explaining the following lemma.
\begin{lemma}\label{lemma codim}
Let $f:X\to B$ be a Lagrangian fibration of a projective irreducible symplectic manifold $X$. Let $X_0=\bigcup_{j\in J} X_j$ where $J=\left\{1, \ldots, r\right\}$ and let $i:Y=\left(X_t\right)_\red \into X$ for $t\in D_0 \subset B$ be the reduction of a general singular fiber contained in $X_0$. Then 
\[
\rk\left(H^2(X,\C)\to[j^*] H^2(\tY,\C)\right)\geq r,
\]
where $\nu:\tY\to Y$ is the normalization and $j = \nu \circ i$. More precisely, the subspace of $H^2(X,\C)$ generated by the classes of the divisors $X_j$ maps onto a subspace of dimension $\geq r-1$ not containing the class of the ample divisor.
\end{lemma}
\begin{proof}
Let $C \subset B$ be a curve obtained by the intersecting $n-1$ general very ample divisors on $B$ and consider the fiber product $X_C=X\times_{B} C$. As $B$ is normal, $X_C$ is smooth. As $t\in D_0$ is general, there is such a curve $C$ with $t\in C$. Let $H$ be a very ample divisor on $X$ and let $H_1,\ldots, H_{n-1}\in \left|H\right|$ be general. Then the intersection $S=X_C \cap H_1\cap\ldots\cap H_{n-1}$ is a smooth surface by Bertini's theorem. By construction it comes with a morphism $g:S\to C$. 

Consider the diagram 
\begin{equation}\label{h2 diagram}
\begin{xy}
\xymatrix{
H^2(X,\C) \ar[r]^{j^*}\ar[d]_\vrho & H^2(\tY,\C) \ar[d]^{\vrho_Y} \\
H^2(S,\C) \ar[r]^{j_S^*} & H^2(\wt{F},\C) \\
}
\end{xy}
\end{equation}
where $F = Y\cap H_1\cap\ldots\cap H_{n-1} \subset S$ and $\wt{F}\to F$ is the normalization. Note that $\tY$ is smooth by \cite[Thm 1.3]{HO07} and $\wt F$ is smooth, as $F$ is a curve. Let $Y=\bigcup_{i=1}^s Y_i$ and $F=\bigcup_{\lambda=1}^q F_\lambda$ be decompositions into irreducible components where $s=\# I$. We put
$F(i):= Y_i \cap H_1\cap\ldots\cap H_{n-1} = \bigcup_{\lambda \in \Lambda_i} F_\lambda$,
where $\Lambda_i \subset \Lambda :=\left\{1,\ldots,q\right\}$ is the subset of all $\lambda$ such that $F_\lambda \subset Y_i$. If the $H_k$ are general enough, $\Lambda$ is the disjoint union of the $\Lambda_i$. 

We will show that the subspace $V \subset H^2(X,\C)$ spanned by the $X_j$ and $H$ maps surjectively onto an $r$-dimensional subspace in $H^2(\wt F,\C)$. This would imply the claim by diagram \eqref{h2 diagram}.

Write $X_0=\sum_j n_j X_j$ and $X_t=\sum_i n_{j(i)} Y_i$ as cycles, where as above $j(i)$ is the unique $j\in J$ with $Y_i \subset X_j$. Recall that $\Lambda = \coprod_i \Lambda_i$ is a disjoint union. So $n_\lambda := n_{j(i)}$ for $\lambda \in \Lambda_i$
is well-defined and we have
$F = \sum_\lambda n_\lambda F_\lambda$.
As $F=\bigcup_{\lambda=1}^q F_\lambda$, we obtain $\wt F=\bigcup_{\lambda=1}^q \wt F_\lambda$ where $\wt F_\lambda$ is the normalization of $F_\lambda$. Thus,
\[
H^2(\wt{F},\C)\isom \bigoplus_{\lambda = 1}^q H^2(\wt{F}_\lambda,\C)\isom \C^q.
\]
If we denote the intersection pairing on $S$ by $(\cdot,\cdot)_S$, then under this isomorphism $j_S^*: H^2(S,\C) \to H^2(\wt{F},\C)$ is given by 
\[
\alpha \mapsto \left((\alpha, F_1)_S,\ldots, (\alpha, F_q)_S\right).
\]
Let $\left\{x_\lambda\mid \lambda \in \Lambda\right\}\subset H^2(\wt{F},\C)^\vee$ be the dual basis of the basis of $H^2(\wt{F},\C)$ obtained corresponding to the standard basis of $\C^q\isom H^2(\wt{F},\C)$. By Zariski's Lemma \cite[Ch III, Lem 8.2]{BPVH} the subspace $W\subset H^2(S,\C)$ spanned by the classes of the $F_\lambda$ maps surjectively to the hyperplane of $\C^q$ given by $\sum_\lambda n_\lambda x_\lambda=0$, So the subspace of $H^2(S,\C)$ spanned by the classes of the $F_\lambda$ and $H\vert_S$ maps surjectively onto $\C^q$. We have
$\vrho_Y(j^*X_j)= j_S^*\vrho(X_j) = \left(\left(\vrho(X_j), F_\lambda\right)_S\right)_\lambda.$
As the $\Lambda_i$ are mutually disjoint, so are the $\Lambda_j:= \bigcup_{j(i) = j} \Lambda_i$. We see from
$\left(\vrho(X_j), F_\lambda\right)_S = \sum_{\mu \in \Lambda_j}(F_\mu, F_\lambda)_S$
that the subspace of $H^2(X,\C)$ generated by the $X_j$ surjects onto a subspace of $\C^q$ of dimension $\geq r-1$. The claim follows as the image of $V$ does not contain $j_S^*(H\vert_S)$.
\end{proof}
\begin{Cor}\label{cor codim}
In the situation of the preceding lemma $\codim M_Y \; \geq r$.
\end{Cor}
\begin{proof} This follows from Theorem \ref{thm vormain} and Lemma \ref{lemma codim}.
\end{proof}
Note that there is no need for a normal crossing hypothesis in the corollary as we only prove an estimate and no equality. The codimension of $M_Y$ is thus bounded by the number of irreducible components of $X_0=f^{-1}(D_0)$ whereas the number of irreducible components of $Y$ does not a priori play a role. Hence, a very interesting and important question is the following
\begin{Que}
Let $Y=\cup_{i\in I} Y_i$ and $X_0=\cup_{j\in J} X_j$ as in the beginning of section \ref{subsec codimension}. Is then $\# I= \# J$? Do we always have $\codim_M M_Y  = \# J$ for normal crossing $Y$?
\end{Que}
There is no obvious reason, why these numbers should be equal, but in all examples we know they are equal. 
Recall that general singular fibers have been classified by Hwang-Oguiso according to their characteristic 1-cycle: this is an effective 1-cycle on a fiber $Y\subset X$, possibly an infinite sum of curves. It was shown to be of Kodaira type or of type $I_\infty$, see \cite[Theorem 1.4]{HO07}  and \cite[Theorem 2.4]{HO11}. The type of a singular fiber will be the type of its characteristic 1-cycle.
\begin{theorem}\label{theorem fibration}
Let $X$ be an irreducible symplectic manifold and let $f:X\to B$ be a Lagrangian fibration. Then $X$ can be deformed, keeping the fibration, to an irreducible symplectic manifold $X'$ with a Lagrangian fibration $f':X'\to B'$ such that outside a codimension $2$ subset $Z\subset B'$, all singular fibers of $f'$ over the complement of $Z$ have a characteristic cycle of Kodaira type I, II, III or IV and such that for every irreducible component $D_0$ of the discriminant divisor the preimage $X'_{D_0}=(f')^{-1}(0)$ is irreducible.
\end{theorem}
\begin{proof}
Let $D_0$ be an irreducible component of the discriminant divisor and let $Y=(X_t)_\red$ for $t\in D_0$ be a general singular fiber. By Lemma \ref{lemma my in mf}, the space $M_Y$ is contained in $\Def(X,f)$. As the fiber type of a singular fiber is generically constant along an irreducible component of the discriminant divisor, it suffices to show that $M_Y \subsetneq \Def(X,f)$ if $X_t$ is not of type I-IV. But for all other Kodaira fibers, there are irreducible components of $X_t$ with different multiplicities. Fiber components with different multiplicities lie in different components of $X_0=f^{-1}(D_0)$, hence $\codim M_Y \geq 2$ by Corollary \ref{cor codim}. By \cite[Corollary 1.2]{Matsus09}, $\codim \Def(X,f)=1$ so that $M_Y\subsetneq \Def(X,f)$ and we conclude the proof.
\end{proof}
\begin{example}\label{example k3}
In the case of K3 surfaces, the situation becomes easier. For elliptic K3 surfaces it is not difficult to see that $\codim M_Y=\# I= \# J$ of irreducible components of the reduction $Y$ of a fiber, if the latter has normal crossings, and $\codim M_Y \geq \# I$ in all other cases, see \cite[Thm VII.3.8]{mydiss}. The analogue of the Hodge-de Rham spectral sequence for $\tOm_Y$ does not degenerate at $E_1$ if $Y$ does not have normal crossings, but one can show that for infinitesimal deformations $\sY\to S$ the $H^q(\tOm_{\sY/S}^p)$ are free $\sO_S$-modules if one uses the differentials in the spectral sequence. With this at hand, one deduces as in the normal crossing case that $\Def^\lt(i)$ and $M_Y$ are smooth. So $M_Y \subset M_Y'$ in all cases.
Consequently, using Theorem \ref{theorem fibration} any elliptic K3 surface can be deformed as a fibration to an elliptic K3 with only nodal and cuspidal curves as singularities. There are examples, where a cuspidal rational curve $Y$ deforms into two nodal curves so that we have $M_Y \subsetneq M_Y'$. 
\end{example}
\subsection{Vista}\label{sec generalizations}
There are several results assuming the general singular fibers to be of a special kind, see \cite{HO10}, \cite{Sa08}, \cite{Sa12}, \cite{Th}. If we knew that complicated general singular fibers only show up in higher codimension in $M$, we could always deform to such special situations.



\begin{thebibliography}{BHPVdV04}

\bibitem[Bea83]{Be83}
Arnaud Beauville.
\newblock Vari\'et\'es {K}\"ahleriennes dont la premi\`ere classe de {C}hern
  est nulle.
\newblock {\em J. Differential Geom.}, 18(4):755--782 (1984), 1983.

\bibitem[BHPV]{BPVH}
Wolf~P. Barth, Klaus Hulek, Chris A.~M. Peters, and Antonius Van~de Ven.
\newblock {\em Compact complex surfaces}, 
\newblock Springer-Verlag, Berlin, second edition, 2004.

\bibitem[Bog78]{Bo78}
F.~A. Bogomolov.
\newblock Hamiltonian {K}\"ahlerian manifolds.
\newblock {\em Dokl. Akad. Nauk SSSR}, 243(5):1101--1104, 1978.


\bibitem[Cam06]{Ca06}
Fr{\'e}d{\'e}ric Campana.
\newblock Isotrivialit\'e de certaines familles k\"ahl\'eriennes de
  vari\'et\'es non projectives.
\newblock {\em Math. Z.}, 252(1):147--156, 2006.

\bibitem[Del68]{De68}
Pierre Deligne.
\newblock Th\'eor\`eme de {L}efschetz et crit\`eres de d\'eg\'en\'erescence de
  suites spectrales.
\newblock {\em Inst. Hautes \'Etudes Sci. Publ. Math.}, (35):259--278, 1968.

\bibitem[Del71]{De71}
Pierre Deligne.
\newblock Th\'eorie de {H}odge. {II}.
\newblock {\em Inst. Hautes \'Etudes Sci. Publ. Math.}, (40):5--57, 1971.

\bibitem[Del74]{De74}
Pierre Deligne.
\newblock Th\'eorie de {H}odge. {III}.
\newblock {\em Inst. Hautes \'Etudes Sci. Publ. Math.}, (44):5--77, 1974.

\bibitem[FK87]{FK}
Hubert Flenner and Siegmund Kosarew.
\newblock On locally trivial deformations.
\newblock {\em Publ. Res. Inst. Math. Sci.}, 23(4):627--665, 1987.

\bibitem[GHJ]{GHJ}
M.~Gross, D.~Huybrechts, and D.~Joyce.
\newblock {\em Calabi-{Y}au manifolds and related geometries}.
\newblock Universitext. Springer-Verlag, Berlin, 2003.

\bibitem[Fri83]{Fr}
R. Friedman.
\newblock Global smoothings of varieties with normal crossings.
\newblock {\em Ann. of Math. (2)}, 118(1):75--114, 1983.

\bibitem[GLR14]{GLR}
Daniel Greb, Christian Lehn, and S\"onke Rollenske.
\newblock Lagrangian fibrations on hyperk\"ahler fourfolds,
\newblock {\em Izvestiya: Mathematics}, 78 (2014), no. 1.

\bibitem[GPR94]{SCVVII}
H.~Grauert, Th. Peternell, and R.~Remmert, editors.
\newblock {\em Several complex variables {VII}, Sheaf-theoretical methods in
  complex analysis}, 
\newblock Springer-Verlag, Berlin, 1994.

\bibitem[GR77]{GR}
H.~Grauert and R.~Remmert.
\newblock {\em Theorie der {S}teinschen {R}\"aume}, 
\newblock Springer-Verlag, Berlin, 1977.

\bibitem[Gra62]{Gra62}
Hans Grauert.
\newblock \"{U}ber {M}odifikationen und exzeptionelle analytische {M}engen.
\newblock {\em Math. Ann.}, 146:331--368, 1962.

\bibitem[HO09]{HO07}
Jun-Muk Hwang and Keiji Oguiso.
\newblock Characteristic foliation on the discriminant hypersurface of a
  holomorphic {L}agrangian fibration.
\newblock {\em Amer. J. Math.}, 131(4):981--1007, 2009.

\bibitem[HO10]{HO10}
Jun-Muk Hwang and Keiji Oguiso.
\newblock Local structure of principally polarized stable lagrangian
  fibrations, 2010.
\newblock preprint {\tt arXiv:1007.2043}.

\bibitem[HO11]{HO11}
Jun-Muk Hwang and Keiji Oguiso.
\newblock Multiple fibers of holomorphic Lagrangian fibrations.
\newblock {\em Commun. Contemp. Math.} 13 (2011), no. 2, 309--329.  

\bibitem[Huy99]{Hu99}
Daniel Huybrechts.
\newblock Compact hyper-{K}\"ahler manifolds: basic results.
\newblock {\em Invent. Math.}, 135(1):63--113, 1999.

\bibitem[Hwa08]{Hw}
Jun-Muk Hwang.
\newblock Base manifolds for fibrations of projective irreducible symplectic
  manifolds.
\newblock {\em Invent. Math.}, 174(3):625--644, 2008.

\bibitem[Kaw92]{Ka92}
Yujiro Kawamata.
\newblock Unobstructed deformations. {A} remark on a paper of {Z}. {R}an:
  ``{D}eformations of manifolds with torsion or negative canonical bundle''.
\newblock {\em J. Algebraic Geom.}, 1(2):183--190, 1992.

\bibitem[Kaw97]{Ka97}
Yujiro Kawamata.
\newblock Erratum on: ``{U}nobstructed deformations. {A} remark on a paper of
  {Z}. {R}an: `{D}eformations of manifolds with torsion or negative canonical
  bundle'\,''.
\newblock {\em J. Algebraic Geom.}, 6(4):803--804, 1997.

\bibitem[KM98]{KM}
J{\'a}nos Koll{\'a}r and Shigefumi Mori.
\newblock {\em Birational geometry of algebraic varieties},
\newblock Cambridge University Press, Cambridge, 1998.

\bibitem[Le11]{mydiss}
Christian Lehn.
\newblock {\em Symplectic Lagrangian Fibrations}.
\newblock Dissertation, Johannes--{G}utenberg--{U}niversit\"at Mainz, 2011.

\bibitem[Le12]{CL12}
Christian Lehn.
\newblock Normal crossing singularities and Hodge theory over Artin rings.
\newblock 2012.
\newblock preprint 


\bibitem[Mat99]{Matsus99a}
Daisuke Matsushita.
\newblock On fibre space structures of a projective irreducible symplectic
  manifold.
\newblock {\em Topology}, 38(1):79--83, 1999.

\bibitem[Mat00]{Matsus00}
Daisuke Matsushita.
\newblock Equidimensionality of {L}agrangian fibrations on holomorphic
  symplectic manifolds.
\newblock {\em Math. Res. Lett.}, 7(4):389--391, 2000.

\bibitem[Mat01]{Matsus01a}
Daisuke Matsushita.
\newblock Addendum: ``{O}n fibre space structures of a projective irreducible
  symplectic manifold''.
\newblock {\em Topology}, 40(2):431--432, 2001.

\bibitem[Mat03]{Matsus03}
Daisuke Matsushita.
\newblock Holomorphic symplectic manifolds and {L}agrangian fibrations.
\newblock {\em Acta Appl. Math.}, 75(1-3):117--123, 2003.

\bibitem[Mat09]{Matsus09}
Daisuke Matsushita.
\newblock On deformations of lagrangian fibrations.
\newblock preprint {\tt arXiv:0903.2098}.

\bibitem[PS08]{PS}
Chris A.~M. Peters and Joseph H.~M. Steenbrink.
\newblock {\em Mixed {H}odge structures},
\newblock Springer-Verlag, Berlin, 2008.

\bibitem[Ra92Def]{Ra92Def}
Ziv Ran.
\newblock Deformations of manifolds with torsion or negative canonical bundle.
\newblock {\em J. Algebraic Geom.}, 1(2):279--291, 1992.

\bibitem[Ra92Lif]{Ra92Lif}
Ziv Ran.
\newblock Lifting of cohomology and unobstructedness of certain holomorphic
  maps.
\newblock {\em Bull. Amer. Math. Soc.}, 26(1):113--117, 1992.

\bibitem[Saw08]{Sa08}
Justin Sawon.
\newblock On the discriminant locus of a {L}agrangian fibration.
\newblock {\em Math. Ann.}, 341(1):201--221, 2008.

\bibitem[Saw12]{Sa12}
Justin Sawon.
\newblock A finiteness theorem for Lagrangian fibrations.
\newblock preprint {\tt arXiv:1212.6470}.

\bibitem[Saw15]{Sa15}
Justin Sawon.
\newblock Singular fibres of generic Lagrangian fibrations.
\newblock unpublished manuscript.

\bibitem[Ser06]{Se}
Edoardo Sernesi.
\newblock {\em Deformations of algebraic schemes}, 
\newblock Springer-Verlag, Berlin, 2006.

\bibitem[Thi08]{Th}
Christian Thier.
\newblock {\em On the Monodromy of 4-dimensional Lagrangian Fibrations}.
\newblock Dissertation, Albert--{L}udwigs--{U}niversit\"at Freiburg, 2008.

\bibitem[Tia87]{Ti}
Gang Tian.
\newblock Smoothness of the universal deformation space of compact
  {C}alabi-{Y}au manifolds and its {P}etersson-{W}eil metric.
\newblock In {\em Mathematical aspects of string theory}, pages 629--646. World Sci. Publishing, Singapore, 1987.

\bibitem[Tod89]{To89}
Andrey~N. Todorov.
\newblock The {W}eil-{P}etersson geometry of the moduli space of {${\rm
  SU}(n\geq 3)$} ({C}alabi-{Y}au) manifolds. {I}.
\newblock {\em Comm. Math. Phys.}, 126(2):325--346, 1989.

\bibitem[Vo92]{Vo92}
Claire Voisin.
\newblock Sur la stabilit\'e des sous-vari\'et\'es lagrangiennes des
  vari\'et\'es symplectiques holomorphes.
\newblock In {\em Complex projective geometry}, pages 294--303. Cambridge Univ. Press, 1992.

\bibitem[Vo1]{Vo1}
Claire Voisin.
\newblock {\em Hodge theory and complex algebraic geometry. {I}},
\newblock Cambridge University Press, Cambridge, 2002.

\bibitem[Vo2]{Vo2}
Claire Voisin.
\newblock {\em Hodge theory and complex algebraic geometry. {II}}, 
\newblock Cambridge University Press, Cambridge, 2003.

\end{thebibliography}
\end{document}